\documentclass{amsart}
\usepackage{amsmath, amssymb,epic,graphicx,mathrsfs,enumerate}
\usepackage[all]{xy}
\usepackage{color}
\usepackage{comment}

\usepackage{amsthm}
\usepackage{amssymb}
\usepackage{latexsym}
\usepackage{longtable}
\usepackage{epsfig}
\usepackage{amsmath}
\usepackage{hhline}

\DeclareMathOperator{\inn}{Inn} \DeclareMathOperator{\perm}{Sym}
 \DeclareMathOperator{\soc}{soc}
\DeclareMathOperator{\aut}{Aut}

 \DeclareMathOperator{\frat}{Frat}

\DeclareMathOperator{\sym}{Sym}

\DeclareMathOperator{\alt}{Alt}
\DeclareMathOperator{\End}{End}

\newcommand{\FF}{\mathbb F}

\renewcommand{\emptyset}{\varnothing}

\newtheorem{thm}{Theorem}
\newtheorem{cor}[thm]{Corollary}
 \newtheorem{lemma}[thm]{Lemma}
\newtheorem{prop}[thm]{Proposition} 
 
\newtheorem{question}[]{Question}

\numberwithin{equation}{section}

\renewcommand{\footnote}{\endnote}
\newcommand{\ignore}[1]{}\makeglossary

\begin{document}
	\bibliographystyle{amsplain}
	\title[Join graph and nilpotent groups]{Finite  groups with the same join graph\\ as a finite nilpotent group}

	\author{Andrea Lucchini}
	\address{Andrea Lucchini\\ Universit\`a degli Studi di Padova\\  Dipartimento di Matematica \lq\lq Tullio Levi-Civita\rq\rq\\ Via Trieste 63, 35121 Padova, Italy\\email: lucchini@math.unipd.it}

	\begin{abstract} Given a finite group $G,$ we denote by $\Delta(G)$ the graph whose vertices are the proper subgroups of $G$ and in which two vertices $H$ and $K$ are joined by an edge if and only if $G=\langle H,K\rangle.$ We prove that if there exists a finite nilpotent group $X$ with $\Delta(G)\cong \Delta(X),$ then $G$ is supersoluble.	\end{abstract}
	\maketitle

Given a finite group $G,$ we denote by $\Delta(G)$ the graph whose vertices are the proper subgroups of $G$ and in which two vertices $H$ and $K$ are joined by an edge if $G$ is generated by $H$ and $K,$ that is, $G=\langle H,K\rangle.$ This graph was introduced in \cite{join} and is called the join graph of $G.$ Notice that the subgroups contained in the Frattini subgroup of $G$ correspond to isolated vertices of $\Delta(G)$, so in particular $\Delta(G)$ contains no edge if $G$ is cyclic of prime-power order.

\

A typical question that arises whenever a graph is associated with a group is the following:

\begin{question}\label{uno}
How similar are the structures of two finite groups $G_1$ and $G_2$ if the graphs $\Delta(G_1)$ and $\Delta(G_2)$ are isomorphic?
\end{question}
We will say that $H\leq G$ is a maximal-intersection in $G$ if there exists a family $M_1,\dots,M_t$ of maximal subgroups of $G$ with $H=M_1\cap \dots \cap M_t.$ Let $\mathcal M(G)$ be the subposet of the subgroup lattice of $G$ consisting of $G$ and all the maximal-intersections in $G$. Notice that $\mathcal M(G)$ is a lattice in which the meet of two elements $H$ and $K$ coincides with their intersection and their join is the smallest maximal-intersection in $G$ containing $\langle H, K\rangle$ (in general $\langle H, K\rangle$ is not a maximal-intersection, see the example at the end of section \ref{facile}).
The maximum element of $\mathcal M(G)$ is $G,$ the minimum element coincides with the Frattini subgroup $\frat(G)$ of $G.$ The role played by $\mathcal M(G)$ in investigating the property of the graph $\Delta(G)$ is clarified by the following proposition.

\begin{prop}\label{equiv}
Suppose that $G_1$ and $G_2$ are finite groups. If the graphs $\Delta(G_1)$ and $\Delta(G_2)$ are isomorphic, then also the lattices $\mathcal M(G_1)$ and $\mathcal M(G_2)$ are isomorphic. 
\end{prop}

Notice that the condition $\mathcal M(G_1)\cong \mathcal M(G_2)$ is necessary but not sufficient to ensure $\Delta(G_1)\cong \Delta(G_2)$. For example consider $G_1=A \times \langle x \rangle $ and $G_2=\perm(3) \times \langle y \rangle,$
where $A\cong C_3\times C_3,$ $\langle x\rangle \cong C_2$ and $\langle y\rangle \cong C_3.$ Let $a_1, a_2, a_3, a_4$ and $b_1, b_2, b_3, b_4$ be generators for the four different non-trivial proper subgroups of, respectively, $A$ and $\perm(3).$ The map sending $A$ to $\perm(3)$  and $\langle a_i, x\rangle$  to $\langle b_i, y \rangle$ for
$1\leq i \leq 4$ induces an isomorphism between $\mathcal M(G_1)$ and $\mathcal M(G_2)$, however all the subgroups of $G_1$ are maximal-intersections, while $\langle (1,2,3)y\rangle$ and $\langle (1,2,3)y^2\rangle$ are not maximal-intersections in $G_2$. In particular $\Delta(G_1)$ has 12 vertices and $\Delta(G_2)$ has 14 vertices. So the following variation of Question \ref{uno} arises. 

\begin{question}\label{due}
How similar are the structures of two finite groups $G_1$ and $G_2$ if the lattices $\mathcal M(G_1)$ and $\mathcal M(G_2)$ are isomorphic?
\end{question}

Our aim is to start to investigate Question \ref{uno} and Question \ref{due}, considering the particular case when $G_1$ is a finite nilpotent group. Notice that if $G_1$ is a finite nilpotent group and $\Delta(G_1)\cong \Delta(G_2)$, then $G_2$ is not necessarily nilpotent. For example if $p$ is an odd prime, $C_p$ is the cyclic group of order $p$ and $D_{2p}$ is  the dihedral group of order $2p,$ then the subgroup lattices of $C_p\times C_p$ and $D_{2p}$ are isomorphic and therefore $\Delta(C_p\times C_p) \cong \Delta(D_{2p}).$ Our main result is the following.

\begin{thm}\label{main}
	Let $G$ be a finite group. If there exists a finite nilpotent group $X$ with $\mathcal M(G)\cong \mathcal M(X),$ then $G$ is supersoluble. 
\end{thm}

\begin{cor}\label{maincor}
	Let $G$ be a finite group. If there exists a finite nilpotent group $X$ with $\Delta(G) \cong \Delta(X),$ then $G$ is supersoluble. 
\end{cor}

Let $\mathfrak M$ be the family of the finite  groups $G$ with the property that $\mathcal M(G)\cong \mathcal M(X)$ for some finite nilpotent group $X.$ In a similar way let  $\mathfrak D$ be the family of the finite  groups $G$ with the property that $\Delta(G)\cong \Delta(X)$ for some finite nilpotent group $X.$ 
By Theorem \ref{main}, if $G\in \mathfrak M,$ then $G$ is supersoluble, but there exist supersoluble groups which do not belong to  $\mathfrak M$
and it is not easy to give a complete characterization of the finite groups in
$\mathfrak M$ or in $\mathfrak D.$ We give a solution of this problem in the particular case when $G$ is a finite group with $\frat(G)=1.$
Recall that a finite group $G$ is called a $P$-group if $G$ is either an elementary abelian $p$-group or is the semidirect product of an elementary abelian normal subgroup of order $p^{n}$ by a group of prime order $q\neq p$ which induces a non trivial power automorphism on $A.$ 

 \begin{prop}\label{ddd}
 	Let $G$ be a finite group with $\frat(G)=1.$ Then  $G\in \mathfrak D$ if and only if $G$ is a direct product of $P$-groups with pairwise coprime orders. 
 \end{prop}

The classification of the Frattini-free  groups in $\mathfrak M$ is more difficult. First we need a definition. Let $t\geq 2$ and $p_1,\dots,p_t$ be prime numbers with the property that $p_{i+1}$ divides $p_i-1$ for $1\leq i\leq t-1.$ We denote by $\Lambda(p_1,\dots,p_t)$ the set of the direct products $H_1\times \dots \times H_{t-1}$, where $H_i\cong C_{p_i}^{n_i}\rtimes C_{p_i+1}$  is
a nonabelian $P$-group. 
Moreover we will denote by $\Lambda^*(p_1,\dots,p_t)$ the direct products
 $X\times Y$ with $X\cong C_{p_1}$ 
 and $Y\in \Lambda(p_1,\dots,p_t).$  Finally let 
$\Lambda$ (respectively $\Lambda^*$) be the union of all the families $\Lambda(p_1,\dots,p_t)$ (respectively $\Lambda^*(p_1,\dots,p_t)$), for any possible choice of $t$ and  $p_1,\dots,p_t.$

\begin{prop}\label{mmm}
	Let $G$ be a finite group with  $\frat(G)=1.$  Then $G\in \mathfrak M$ if and only if $G$ is a direct product $H_1\times \dots \times H_u$, where the orders of the factors are pairwise coprime and each of the factors is of one of the following types:
	\begin{enumerate}
		\item an elementary abelian $p$-group;
		\item a group in $\Lambda;$
		\item a group in $\Lambda^*.$
	\end{enumerate}
\end{prop}	
It follows from the previous proposition that $\perm(3)\times C_2$ is an example (indeed the one of smallest possible order) of a supersoluble group $G$ which does not belong to $G\in \mathfrak M.$

Notice that our proof of Theorem \ref{main} uses the classification of the finite simple groups. Theorem \ref{main} is invoked in the proof of Proposition \ref{mmm}, which therefore in turn depends on the classification. On the contrary, Proposition \ref{ddd} can be directly proved without using Theorem \ref{main} and the  classification of the finite simple groups. Indeed it turns out that if $G\in \mathfrak D$ and $\frat(G)=1,$ then $G$ has the same subgroup lattice of a finite abelian group, and the groups with these property have been classified by R. Baer \cite{baer}. However, we are not able to deduce  Corollary \ref{maincor} from Proposition
\ref{ddd}, so also our proof of this result depends on the classification. To avoid the use of the classification in the proof of  Corollary \ref{maincor}, one should give a positive answer to the following question, that we leave open. Is it true that $\Delta(G_1)\cong \Delta(G_2)$ implies $\Delta(G_1/\frat(G_1))\cong \Delta(G_2/\frat(G_2))?$ The obstacle in dealing with this question, is that it is not clear whether and how one can deduce which vertices of the graph $\Delta(G)$ correspond to subgroups of $G$ containing $\frat(G).$

\section{Preliminary results}\label{facile}

 Denote by $\mathcal N_G(X)$ the set of neighbourhoods of the vertex $X$ in the graph $\Delta(G).$ We define an equivalent relation $\equiv_G$ by the rules $X\equiv_G Y$ if and only if $\mathcal N_G(X)=\mathcal N_G(Y).$  If $X \leq G,$ let $\tilde X$ be the intersection of the maximal subgroups of $G$ containing $X$ (setting $\tilde X=G$ if no maximal subgroup of $G$ contains $X)$.

\begin{lemma}\label{nxny} $\mathcal N_G(X)\subseteq \mathcal N_G(Y)$ if and only if $\tilde X \leq \tilde Y.$ In particular
	$X\equiv_G Y$ if and only if $\tilde X=\tilde Y.$	
\end{lemma}

\begin{proof}
	Assume $\mathcal N_G(X)\subseteq \mathcal N_G(Y)$ and let $M$ be a maximal subgroup of $G$. If $Y\leq M,$ then $\langle Y, M\rangle \neq G,$ so $M\notin \mathcal N_G(Y)$ and consequently $M\notin \mathcal N_G(X)$ i.e. $\langle X, M\rangle \neq G:$ this implies $X\leq M.$ If follows $\tilde X\leq \tilde Y.$ Conversely, assume $\tilde X \leq \tilde Y,$ or equivalently that every maximal subgroup of $G$ containing $Y$ contains also $X.$ If $Z\notin \mathcal N_G(Y),$ then $\langle Y, Z\rangle \leq M$ for some maximal subgroup $M$ of $G.$ It follows $\langle X, Z\rangle \leq M$ and consequently $Z\notin \mathcal N_G(X).$ 
\end{proof}

\begin{prop}Let $G$ be a finite group. The lattice $\mathcal M(G)$ can be completely determined from the knowledge of the graph $\Delta(G).$
\end{prop}
\begin{proof}
For an equivalent class $\mathcal C$ for the relation $\equiv_G,$ set $H_{\mathcal C}=\langle X \mid X\in \mathcal C\rangle.$ It follows from Lemma \ref{nxny} that $H_\mathcal C=\tilde X$ for every $\tilde X\in \mathcal C.$ So the map $\phi: \mathcal C \mapsto H_{\mathcal C}$ induces a bijection from the set of the equivalence classes to the set of the maximal-intersections in $G.$ Moreover, if $X_1, X_2\in \mathcal M(G),$ then $X_1\leq X_2$ if and only if  $\mathcal N_G(X_1) \subseteq \mathcal N_G(X_2).$
\end{proof}

We conclude this section with an example  showing that if $X_1, X_2 \in \mathcal M(G),$ then it is not necessarily true that $\langle X_1, X_2 \rangle \in \mathcal M(G).$
Let $\FF$ be the field with $3$ elements and let $C=\langle -1 \rangle $ be the multiplicative group of $\FF.$
Let $V=\FF^3$
be a $3$-dimensional vector space over $\FF$ and let $\sigma=(1,2,3)\in \perm(3).$  The wreath product $H=C\wr \langle \sigma \rangle$ has an irreducible action on $V$ defined as follows:
if $v=(f_1,f_2,f_3)\in V$ and $h=(c_1,c_2,c_3)\sigma^i \in H$, then $v^h=(f_{1\sigma^{-i}}c_{1\sigma^{-i}},f_{2\sigma^{-i}}c_{2\sigma^{-i}},f_{3\sigma^{-i}}c_{3\sigma^{-i}}).$
Consider the semidirect product
$G=V\rtimes H$ and let $v=(1,-1,0)\in V.$ Since $H$ and $H^v$ are two maximal subgroups of $G,$ $K:=H\cap H^v=C_H(v)=\{(1,1,z)\mid z \in C\}\cong C_2$ is a maximal-intersection in $G.$
Since $\frat(H)=1,$ $V$ is also a maximal-intersection in $G.$ However $VK$ is not a maximal-intersection in $G.$ Indeed if $V\leq X$ is a maximal-intersection in $G,$ then $X=VY$ with $Y$ a maximal-intersection in $H.$ But $H\cong C_2\times \alt(4)$ and the unique subgroup of order 2 of $H$ that can be obtained as intersection of maximal subgroups is $\{(z,z,z)\mid z \in C\}.$

\section{Proof of Theorem \ref{main}}

Recall that the M\"obius function  $\mu _G$ is defined on the subgroup
lattice  of $G$ as $\mu _G(G)=1$ and $\mu _G(H)=-\sum
_{H<K}\mu _G(K)$ for any $H<G$. If $H\leq G$ cannot be expressed as an intersection
of maximal subgroups of $G$, then $\mu_G(H)=0$ (see \cite[Theorem 2.3]{hall}), so for every $H\in \mathcal M(G)$, the value $\mu_G(H)$ can be completely determined from the knowledge of the lattice $\mathcal M(G).$

\begin{prop}\label{moe}
	Let $G$ be a finite soluble group. For every irreducible $G$-module $V$ define
	$q(V)=|\End_{G}(V)|$,
	set $\theta(V)=0$  if $V$ is a trivial $G$-module,  and $\theta(V) = 1$ otherwise,
	and let $\delta(V)$ be the number of  chief factors $G$-isomorphic to $V$ and complemented in an arbitrary chief series of $G$. Let $\mathcal V(G)$ be the set of irreducible $G$-module $V$ with $\delta(V)\neq 0.$
	Then
	$$\mu_G(1)=\begin{cases}
	0 &\text{if \ $\prod_{V\in \mathcal V(G)}|V|^{\delta(A)}\neq |G|,$}\\
\prod_{V\in \mathcal V(G)}(-1)^{\delta(V)}|V|^{\theta(V) \delta(V)}q(V)^{\binom{\delta(V)}{2}} &\text{otherwise.}
	\end{cases}$$
\end{prop}
\begin{proof}
We prove the statement by induction on the order of $G$. Let $N$ be a minimal normal subgroup of $G.$ By \cite[Lemma 3.1]{isa}
$$\mu_G(1)=\mu_{G/N}(1)\sum_{K\in \mathcal K}\mu_G(K),$$
denoting by $\mathcal K$ the set of all subgroups of $G$ which complement $N.$ If $K=\emptyset,$ then $N$ is a non-complemented chief factor of $G$ and $\mu_G(1)=0.$ In any case, since $N$ is a minimal normal subgroup of $G,$ if $K\in \mathcal K,$ then $K$ is a maximal subgroup of $G$ and consequently $\mu_G(K)=-1.$ Thus $\mu_G(1)=-\mu_{G/N}(1)\cdot c,$
where $c$ is the number of complements of $N$ in $G.$ To conclude it suffices to notice that, by \cite[Satz 3]{g2}, $c=|N|^{\theta(N)}q(N)^{\delta(N)-1}.$
\end{proof}

\begin{cor}\label{monil} If $X \cong C_{p_1}^{m_1}\times \dots \times  C_{p_t}^{m_t},$ then $\mu_X(1)=(-1)^{m_1}p_1^{\binom{m_1}{2}}\cdots (-1)^{m_t}p_t^{\binom{m_t}{2}}.$
\end{cor}

\begin{lemma}\label{normali}
	Let $G$ be  a finite group and assume that there exists a finite nilpotent group with $\mathcal M(G)\cong \mathcal M(X).$ Then every normal subgroup $N$ of $G$ containing $\frat(G)$ is a maximal-intersection in $G$ and $\mu_G(N)\neq 0.$
\end{lemma}
\begin{proof}
We have $\mathcal M(G/\frat(G))\cong \mathcal M(G)\cong  M(X)\cong \mathcal M(X/\frat(X))$,
and this implies $\mu_{X/\frat(X)}(1)=\mu_{G/\frat(G)}(1).$
By Corollary  \ref{monil}, $\mu_{X/\frat(X)}(1)\neq 0$ and therefore $\mu_{G/\frat(G)}(1)\neq 0.$ If $N$ is a normal subgroup of $G$ containing $\frat(G),$ then we deduce by  \cite[Lemma 3.1]{isa} that $\mu_G(N)=\mu_{G/N}(1)$ divides $ \mu_{G/\frat(G)}(1)$. As a consequence $\mu_G(N)\neq 0$ and $N$ is a maximal-intersection in $G.$
\end{proof}
 
\begin{lemma}\label{super}
	Let $H$ be a finite supersoluble group and $V$ a faithful irreducible $H$-module. Consider the semidirect product $G=V\rtimes H.$ Suppose that there exists a nilpotent group $X$ with $\mathcal M(G)\cong \mathcal M(X).$ Then $V$ is cyclic of prime order.
\end{lemma}

\begin{proof} Since $\mathcal M(X)\cong \mathcal M(X/\frat(X)),$ we may assume $\frat(X)=1.$ There exist $v$ and $w$ in $V$ such that $C_H(v)\cap C_H(w)=1$ (see \cite[Theorem A]{wolf}). This implies that  $H, H^v, H^w$ are maximal subgroups of $G$ with trivial intersection. But then also $X$ must contain three maximal subgroups with trivial intersection and consequently $|X|$ is the product of at most three (non necessarily  distinct) primes.  Suppose $|V|=p^a,$ with $p$ a prime and $a\geq 2.$ By Proposition \ref{moe}, $0\neq \mu_X(1)=\mu_G(1)$ is divisible by $p^a$. This is possible only if
	$X\cong C_p\times C_p\times C_p$ and $\mu_X(1)=\mu_G(1)=-p^3.$
	Moreover, since $N$ is a minimal element in $\mathcal M(G),$ it must be $\mathcal M(H)\cong \mathcal M(G/N)\cong \mathcal M (C_p\times C_p)$ and therefore, by Corollary \ref{monil},  $\mu_H(1)=p.$ Moreover a chain in $\mathcal M(H)$ must have length at most 2,
	so $|H|$ is the product of two primes. It follows from from Proposition \ref{moe} that $O_p(G)\neq 1$, in contradiction with the fact that $H$ is a faithful irreducible $H$-module.
\end{proof}

\begin{lemma}\label{as}Let $G$ a monolithic primitive group with non-abelian socle. If there exists a finite nilpotent group $X$ with $\mathcal M(G) \cong \mathcal M(X),$ then $G$ is an almost simple group.
\end{lemma}
\begin{proof}

There exists a finite nonabelian simple group $S$
such that $N=\soc(G)= S_1 \times \ldots \times S_n$, with $S_i\cong S$ for
$1\leq i \leq n$. Assume by contradiction $n\geq 2.$ 
 Let  $\psi$ be the map from  $N_L(S_1) $ to  $\aut(S)$ induced by
the conjugacy  action on $S_1$.
Set $H=\psi(N_G(S_1))$ and note that $H$ is an almost simple group with socle
$S=\inn(S)=\psi(S_1)$.
Let $T:=\{t_1,\ldots,t_n\}$
be a right transversal of $N_G(S_1)$ in $G;$
the map $$\phi_T: G \to
H \wr \sym(n)$$ given by
$$g \mapsto ( \psi(t_{1}^{} g t_{1 \pi}^{-1}), \dots ,  \psi(t_n^{} g t_{n
	\pi}^{-1})) \pi $$
where $\pi \in \sym (n)$ satisfies $t_i^{}g t_{i \pi}^{-1} \in
N_G(S_1)$ for all $1\leq i\leq  n$, is an injective
homomorphism.
So we may identify  $G$ with
its image in $H \wr \sym(n)$;  in this identification,  $N$ is contained in
the base subgroup $H^n$ and $S_i $ is a subgroup of the $i$-th
component of $H^n$.
Since $\mathcal M(X)\cong \mathcal M(X/\frat(X)),$ we may assume $\frat(X)=1.$ By Lemma \ref{normali}, $\frat(G/N)=1$ and so there exist $t$ maximal subgroups $M_1,\dots,M_t$ of $G$ such that
$$N=M_1\cap \dots \cap M_t < M_1\cap \dots \cap M_{t-1}<\dots < M_1\cap M_2 < M_1 <G.$$  
Let $R$ be a maximal subgroup of $H$ with $H=RS$ and set $K=R\cap S.$ It must be $K\neq 1$ (see for example the last paragraph of the proof of the main theorem in \cite{lps}). Notice that $L:=G \cap (R \wr \perm(n))$ is a maximal subgroup of $G$ (\cite{classes} Proposition 1.1.44).  We have $D:=L\cap M_1\cap \dots \cap M_t=L\cap N=K^n.$ Choose a subset $\{s_1,\dots,s_m\}$ of $S$ with minimal cardinality 
with respect to the property $K\cap K^{s_1}\cap \dots \cap K^{s_m}=1.$
Set 
$$\begin{aligned}
\alpha_1=&(s_1,\dots,s_1), \alpha_2=(s_2,\dots,s_2),\dots,\alpha_m=(s_m,\dots,s_m),\\
\beta_1=&(s_1,1,\dots,1), \beta_2=(s_2,1,\dots,1),\dots,\beta_m=(s_m,1,\dots,1),\\
\gamma_1=&(1,s_1,\dots,s_1), \gamma_2=(1,s_2,\dots,s_2),\dots,\gamma_m=(1,s_m,\dots,s_m).
\end{aligned}$$
For $1\leq i \leq m,$ set
$$\begin{aligned}
A_i:&=L^{\alpha_i}\cap \dots \cap L^{\alpha_m}\cap D,\\
B_i:&=L^{\beta_i}\cap \dots \cap L^{\beta_m}\cap L^{\gamma_1}\cap \dots \cap L^{\gamma_m}\cap D,\\
C_i:&=L^{\gamma_i}\cap \dots \cap L^{\gamma_m}\cap D.
\end{aligned}$$
We have 
$$1 = A_1 < \dots < A_m < D, \quad 1 = B_1 < \dots < B_m < C_1 < \dots < C_m < D.$$
In particular $$\{M_1,\dots,M_t,L,L^{\alpha_1},\dots,L^{\alpha_m}\}, \quad \{M_1,\dots,M_t,L,L^{\beta_1},\dots,L^{\beta_m},L^{\gamma_1},\dots,L^{\gamma_m}\}$$
	are two families of maximal subgroups of $G$ that are minimal with respect to the property that their intersection is the trivial subgroup. However the assumption $\mathcal M(G)\cong \mathcal M(X)$ implies that all the family of maximal subgroups of $G$ with this property must have the same size.
\end{proof}

\begin{lemma}\label{cinque}
	If $G$ is a finite almost simple group, then there exist $t\leq 5$ maximal subgroups of $G$ with the property that $M_1\cap \dots \cap M_t=1.$
\end{lemma}
\begin{proof}
The result follows from  \cite[Theorem 1]{bu}, except when $S=\soc(G)$ is an alternating group or a classical group and all the  primitive actions of $X$ are of standard type.	
If $\soc(G)$ is of alternating type, then the result follows from \cite[Corollary 1.4, Corollary 1.5, Remark 1.6]{bcam} (see also \cite[Theorem 2]{gl} and its proof). In the case of classical groups, we can build up a non-standard action by taking primitive actions with stabilizer in one of the Aschbacher classes $\mathscr{C}_2$,  $\mathscr{C}_3$,  $\mathscr{C}_4$,  $\mathscr{C}_5$,  $\mathscr{C}_6$,  $\mathscr{C}_7$
(see \cite{kl} Tables 3.5.A. 3.5.B, 3.5.C, 3.5.D, 3.5.E and 3.5.F), except when  $S=\Omega^+_{2p}(2)$ and  $p$ is an odd prime.	In this case, $|G:S|\leq 2$. Let $\tilde V$ be
the natural module for $G$ and let 
be the set of nondegenerate plus-type subspaces of
dimension $k+1.$ Then $G$ acts primitively on this set and by the proof of \cite[Theorem 6.13]{bgl} it contains three maximal subgroups $M_1$, $M_2$, $M_3$ such that $M_1\cap M_2\cap M_3\cap S=1$. so $t\leq 4.$
\end{proof}

\begin{lemma}\label{mono}
If $G$ is a finite monolithic primitive group with non-abelian socle, then there is no finite nilpotent group $X$ with $\mathcal M(G) \cong \mathcal M(X).$
\end{lemma}
\begin{proof}
Assume, by contradiction, that there exists a finite nilpotent group $X$ with $\mathcal M(X) \cong \mathcal M(G).$ By Lemma \ref{as}, $G$ is a finite almost simple group. Moreover, as in the proofs on the previous Lemmas, we may assume $\frat(X)=1$ and consequently $0\neq \mu_X(1)=\mu_G(1).$ By Lemma \ref{cinque}, $G$ contains $t\leq 5$ maximal subgroups with trivial intersection. But then $X$ satisfies the same properties, and consequently $|X|$ is the product of at most $t\leq 5$ primes. It follows from Corollary \ref{monil} that $\mu_X(1)=\mu_G(1)$ is divisible by at most two different primes. By \cite[Theorem 4.5]{isa}, $|G|$ divides $m\cdot \mu_G(1),$ where $m$ is the square-free part of $|G/G^\prime|.$ So, if $S=\soc(G),$ then, since $S\leq G^\prime$, $m$ divides $|G/S|$ and consequently $|S|$ divides $\mu_G(1)=\mu_X(1).$ But then $|S|$ is divisible by at most two different primes, so it is soluble by the Burnside's $p^aq^b$-theorem, a contradiction.
\end{proof}

\begin{proof}[Proof of Theorem \ref{main}] We prove our statement by induction on the order of $G$. If $\frat(G)\neq 1,$ then $\mathcal M(G/\frat(G))\cong \mathcal M(X/\frat(X)),$ so $G/\frat(G)$ is supersoluble by induction. But this implies that $G$ itself is supersoluble. So we may assume $\frat(G)=1.$ 
Assume, by contradiction, that $G$ is not soluble. Then there exists a non-abelian chief factor $R/S$ of $G.$ Let $L=G/C_G(R/S).$	Notice that $L$ is a primitive monolithic group whose socle is isomorphic to $R/S.$ 
By Lemma \ref{normali}, $C_G(R/S)$ is a maximal-intersection in $G$. But then $\mathcal M(L)\cong \mathcal M(X/Y)$ for a suitable normal subgroup $Y$ of $X$, in contradiction with Lemma \ref{mono}.
So we may assume that $G$ is soluble. Assume by contradiction that $G$ is not supersoluble.
Let $1=N_0<N_1<\dots <N=G$ be a chief series of $G$ and let $j$ be the largest positive integer with the property that the chief factor $N_j/N_{j-1}$ is not cyclic. Let $V=N_j/N_{j-1}$ and $H=G/C_G(V).$ By Lemma \ref{normali}, $N_j/N_{j-1}$ is a complemented chief factor of $G$, so there exists a normal subgroup $M$ of $G$ with $G/M\cong V\rtimes H.$
Again by Lemma \ref{normali}, $M$ is a maximal-intersection in $G$, so there exists $Y\leq X$ such that $\mathcal M(G/M)\cong \mathcal M(X/Y).$ By our choice of the index $j$, the factor group $G/N_j$ is supersoluble. Since $N_j\leq C_G(V),$ also $H$ is supersoluble. But then
it follows from Lemma \ref{super} that $V$ is cyclic of prime order, in contradiction with our assumption.
\end{proof}

\section{Frattini-free groups in $\mathfrak D$ and $\mathfrak M$}\label{last}



\begin{proof}[Proof of Proposition \ref{ddd}]
Assume that $X$ is  a finite nilpotent group with $\Delta(X)\cong \Delta(G)$. Since $\frat(G)=1,$ the unique isolated vertex in $\Delta(G)$ is the one corresponding to the identity subgroup. The same must be true in $\Delta(X)$ and therefore $\frat(X)=1.$ Hence $X$ is a direct product of elementary abelian groups. In particular every subgroup of $X$ is a maximal-intersection in $X$, so the lattice $\mathcal M(X)$ coincides with the entire subgroup lattice $\mathcal L(X)$ of $X$. This is equivalent to say that if $Y_1$ and $Y_2$ are different subgroups of $G$, then $\mathcal N_G(Y_1)\neq  \mathcal N_G(Y_2)$. Again, the same property holds for $\Delta(G)$ and consequently $\mathcal M(G)\cong \mathcal L(G).$ So by Proposition \ref{equiv}, $\mathcal L(G)\cong \mathcal L(X)$ and  the conclusion follows from 
\cite[Theorem 2.5.10]{sch}.
\end{proof}

\begin{lemma}\label{mdp}Suppose that $X_1$ and $X_2$ are finite groups. If no simple group is a homomorphic image of both $X_1$ and $X_2$ then $\mathcal M(X_1\times X_2) \cong \mathcal M(X_1) \times \mathcal M(X_2).$
\end{lemma}
\begin{proof}
A maximal subgroup $M$ of a direct product $X_1\times X_2$ is of standard type if either $M=Y_1 \times X_2$ with
	$Y_1$ a maximal subgroup of $X_1$ or $M=X_1 \times Y_2$ with
	$Y_2$ a maximal subgroup of $X_2$.
	A maximal subgroup $M$ of $X_1\times X_2$ is of {diagonal type} if there exist
	a maximal normal subgroup $N_1$ of $X_1$, a maximal normal subgroup $N_2$ of $X_2$ and an isomorphism
	$\phi: X_1/N_1 \to X_2/N_2$ such that $M=\{(x_1,x_2)\in H_1\times H_2 \mid \phi(x_1N_1)=x_2N_2\}.$
	By \cite[Chap. 2, (4.19)]{suzuki}, a maximal subgroup of $X_1\times X_2$ is either of standard type or of diagonal type. If no simple group is a homomorphic image of both $X_1$ and $X_2$ then all the maximal subgroups of $X_1\times X_2$ are of standard type. In particular $K\in \mathcal M(X_1\times X_2)$ if and only if $K=K_1\times K_2,$ with $K_1\in \mathcal M(X_1)$ and $K_2\in \mathcal M(X_2).$ 
\end{proof}
\begin{lemma}\label{inverso}The following hold:
	\begin{enumerate}
\item If $G\in \Lambda(p_1,\dots,p_t)$, then $\mathcal M(G)\cong \mathcal M(C_{p_1}^2\times \dots \times  C_{p_{t-1}}^2).$
\item If $G\in \Lambda^*(p_1,\dots,p_t)$ 
then $\mathcal M(G)\cong \mathcal M(C_{p_1}^2\times \dots \times  C_{p_{t-1}}^2\times C_{p_t}).$ 
\end{enumerate}
\end{lemma}
\begin{proof}
Let $H\cong C_{p}^n\rtimes C_q$ be a nonabelian $P$-group.  By \cite[Theorem 2.2.3]{sch}, the subgroup lattices of $H$ and $C_p^{n+1}$ are isomorphic, and consequently $\mathcal M(H)\cong \mathcal M(C_p^{n+1}).$
 Now assume  $G=H_1\times \dots \times H_{t-1}\in \Lambda(p_1,\dots,p_t)$, with $H_i\cong C_{p_i}^{n_i}\rtimes C_{p_i+1}.$ By Lemma \ref{mdp}, $$\begin{aligned}\mathcal M(G)&\cong\mathcal M(H_1\times\dots\times H_{t-1})
\cong \mathcal M(H_1)\times \dots \times \mathcal M(H_{t-1})\\
&\cong \mathcal M(C_{p_1}^{n_1+1})\times \dots \times \mathcal M(C_{p_{t-1}}^{n_{t-1}+1})\!\cong\! \mathcal M(C_{p_1}^{n_1+1}\times \dots \times  C_{p_{t-1}}^{n_{t-1}+1}).\end{aligned}$$ This proves (1). 
If $G=H_1\times \dots \times H_{t-1}\times C_{p_1}\in \Lambda^*(p_1,\dots,p_t)$ 
with $H_i\cong C_{p_i}^{n_i}\rtimes C_{p_i+1},$ then, again by Lemma \ref{mdp}, 
$$\begin{aligned}\mathcal M(G)&\cong\mathcal M(H_1\times\dots\times H_{t-1}\times C_{p_1})\\&
\cong \mathcal M(H_1)\times \dots \times \mathcal M(H_{t-1})\times \mathcal M(C_{p_1})\\
&\cong \mathcal M(C_{p_1}^{n_1+1})\times \dots \times \mathcal M(C_{p_{t-1}}^{n_{t-1}+1})\times \mathcal M(C_{p_1})\\&\cong \mathcal M(C_{p_1}^{n_1+1})\times \dots \times \mathcal M(C_{p_{t-1}}^{n_{t-1}+1})\times \mathcal M(C_{p_t})\\&
\cong \mathcal M(C_{p_1}^{n_1+1}\times \dots \times  C_{p_{t-1}}^{n_{t-1}+1}\times C_{p_t}).\end{aligned}$$
So (2) is also proved.
\end{proof}

\begin{proof}[Proof of Proposition \ref{mmm}] First we prove by induction on the order of $G$ that if $G\in \mathfrak M,$ then $G$ is as described in the statement. Let $N$ be a minimal normal subgroup of $G$. By Theorem \ref{main}, there exists a prime $p$ such that $N\cong C_p$. Moreover, since $\frat(G)=1,$  $N$ has a complement, say $K$ in $G.$ By Lemma \ref{normali}, $K\cong G/N$ is a Frattini-free group belonging to $\mathfrak M$, so by induction $K=H_1\times \dots \times H_u$, where $H_1,\dots,H_u$ have coprime orders and are as described in the statement. 
	
	First assume that $N$ is central in $G$. If $p$ does not divide the order of $K,$ then $G=H_1\times \dots \times H_u \times N$ is a factorization with the required properties. Otherwise there exists a unique $i$ such that $p$ divides $|H_i|$. It is not restrictive to assume $i=u.$ If $H_u$ is either elementary abelian or $H_u \in \Lambda(p_1,\dots,p_t)$  with $p_1=p$, then we set
	$\tilde H_u=H_u\times C_p$ and the factorization $G=H_1\times \dots \times H_{u-1} \times \tilde H_u$ satisfies the required properties.
In the other cases, there exists a prime $q\neq p$ and a normal subgroup $L$ of $H_u$ such that $J=H_u/L$ is isomorphic either  to $C_q\rtimes C_p$  or to $(C_p\rtimes C_q)\times C_p$.
By Lemma \ref{normali}, $Y=N\times J \cong G/(H_1\times \dots \times H_{u-1}\times L)\in \mathfrak M.$ So there exists a Frattini-free nilpotent group $X$ with $\mathcal M(X) \cong \mathcal M(Y).$ If $J\cong C_q\rtimes C_p,$ then $\mu_X(1)=\mu_Y(1)=-p\cdot q$ and $|X|$ is the product of three primes, but this possibility is
 excluded by Corollary \ref{monil}. If $J\cong (C_p\rtimes C_q)\times C_p,$ then $\mu_X(1)=\mu_Y(1)=p^2,$ again in contradiction with 
 Corollary \ref{monil}.
 
 Now assume that $N$ is not central. By Lemma \ref{normali}, $\frat(G/C_G(N))=1$, so $G/C_G(N)\cong C_q$ where $q$ is a square-free positive integer. Moreover there exists a Frattini-free nilpotent group $X$ such that $\mathcal M(X)\cong \mathcal M(Y)$ with $Y=G/C_K(N) \cong C_p \rtimes C_q.$ The identity subgroup of $Y$ can be obtained as intersection of two conjugated subgroups of order $q$. This implies that $|X|$ is the product of two primes and consequently $\mathcal M(Y)\cong \mathcal M(X)$ cannot contain chains of length $\geq 2.$ But then $q$ is a prime. In particular there exists a unique $i$ such that $q
$ divides $|H_i|.$ It is not restrictive to assume $i=u.$ Notice that $C_q\cong H_u/C_{H_u}(N),$ so $q$ divides $|H_u/H_u^\prime|.$
We distinguish the different possibilities for $H_u.$
	If $H_u=C_q^t,$ then $t=1$ (and so $NH_u\in \Lambda(p,q)$), otherwise $Y\cong (C_p\rtimes C_q)\times C_q$ would be an epimorphic image of $G$ and consequently there would exist a nilpotent group $X$ whose order is the product of three primes such that $\mu_X(1)=\mu_X(Y)=-p\cdot q,$ in contradiction with Corollary \ref{monil}. Assume that $H_u\in  \Lambda(p_1,\dots,p_t),$ with $q=p_i.$ We deduce from the fact that $q$ divides $|H_u/H_u^\prime|$ that $i\geq 2.$
	Let $r=p_{i-1}$ and $R$  a non-central normal subgroup of $H_u$ with order $r$. There exists a subgroup $Q$ of $H_u$ which has order $q$ and does not centralize neither $N$ nor $R.$ The semidirect product
	$Y=(N\times R)\rtimes Q\cong (C_p\times C_r)\rtimes C_q$ is an epimorphic image of $G$ and consequently there  exists a nilpotent $X$ whose order is the product of three primes such that $\mu_X(1)=\mu_Y(1)$ is divisible by $p\cdot r.$ By Corollary \ref{monil}, this is possible only if $p=r, X\cong C_p^3$, $\mu_X(1)=-p^3$ and $N$ and $R$ are $Q$-isomorphic (and consequently $G$-isomorphic). But then $p$ divides $|H_u|$ and $NH_u\in \Lambda(p_1,\dots,p_t).$
	 Assume $H_u\in \Lambda^*(p_1,\dots,p_t).$
	 If $q\neq p_1,$ we may repeat the previous argument and conclude that $p$ divides $|H_u|$ and $NH_u\in \Lambda^*(p_1,\dots,p_t).$
	 If $q=p_1,$ then $NH_u\in \Lambda(p,p_1,\dots,p_t).$
	 We conclude that in any cases one of the following occurs:
	 \begin{enumerate}
	 	\item 
 $NH_u\in \Lambda(p,p_1,\dots,p_t),$ 
\item 	  $NH_u\in \Lambda(p_1,\dots,p_t),$
\item	 $NH_u\in \Lambda^*(p_1,\dots,p_t).$  
\end{enumerate}
 If $p$ does not divide $|H_1|\cdots |H_{u-1}|,$  then the factorization $H_1\times \dots H_{u-1} \times NH_u$ satisfies the requests of the statement. Otherwise  we may assume that $p$ divides $|H_1|.$
	 Notice that in this case $p$ does not divides $H_u$, so $NH_u\in \Lambda(p,p_1,\dots,p_t)$
	  If $H_1$ admit a non-central chief factor of order $p,$ then there exists a prime $r$ such that $Y=(C_p\rtimes C_q) \times (C_p\rtimes C_r)$ is an epimorphic image of $G.$ There would exist a nilpotent group $X$ with $\mu_X(1)=\mu_Y(1).$ However by Proposition \ref{moe}, $\mu_Y(1)=p^2\cdot q^\eta,$ with $\eta=1$ if $q=r,$ $\eta=0$ otherwise, while by Corollary \ref{monil}, $p$ cannot divide $\mu_X(1)$ with multiplicity equal to 2. The only possibility that remains to discuss is $H_1\cong C_p^t.$ If $t\geq 2,$ then $Y=(C_p\rtimes C_q)\times C_p^2$ is an epimorphic image of $G,$ and there would exist a nilpotent group $X$ with $\mu_X(1)=\mu_Y(G)=p^2,$ again in contradiction with Corollary \ref{monil}. But then $t=1$ and $H_1\times NH_u\in \Lambda^*(p_1,p_2,\dots,p_t)$ with $p_1=p.$ Setting $\tilde H_1=H_1\times (NH_u),$ we conclude that $\tilde H_1\times H_2\times \dots \times H_{u-1}$ is the factorization we are looking for.
	
	Conversely, assume that $G=H_1\times \dots \times H_u$ is a factorization with the properties described by the statement. By Lemma \ref{inverso}, for every $1\leq i\leq u,$ there exists a nilpotent group that that $\mathcal M(H_i)=\mathcal M(X_i)$ and  $\pi(X_i)=\pi(H_i).$ But then, by Lemma \ref{mdp}, $\mathcal M(G) \cong \mathcal M(H_1)\times \cdots\times \mathcal M(H_u) \cong \mathcal M(X_1)\times \cdots\times  \mathcal M(X_u).$
\end{proof}

\end{document}